\documentclass[11pt]{amsart}
\usepackage{amsmath,amssymb,mathrsfs}
\usepackage[colorlinks, linkcolor=red, citecolor=green,filecolor=magenta,urlcolor=cyan]{hyperref}

\allowdisplaybreaks

\newcommand{\abs}[1]{\left\lvert#1\right\rvert}
\newcommand{\kh}[1]{\left(#1\right)}%
\newcommand{\zkh}[1]{\left[#1\right]}%
\newcommand{\dkh}[1]{\left\{#1\right\}}%

\newtheorem{theorem}{Theorem}

\newtheorem{lem}[theorem]{Lemma}
\newtheorem{conj}{Conjecture}

\begin{document}

\title[Minimal hypersurface in $\mathbb S^5$]{Minimal hypersurfaces in $\mathbb S^5$ with constant scalar curvature and zero Gauss curvature are totally geodesic}

\author{Qing Cui}
\address{School of Mathematics \\ Southwest Jiaotong University \\ 611756 Chengdu \\ Sichuan \\ China} 
\email{cuiqing@swjtu.edu.cn}

 \begin{abstract}
We show that a closed minimal hypersurface in $\mathbb S^5$ with constant scalar curvature and zero Gauss curvature is totally geodesic.
\end{abstract}
\maketitle

\section{Introduction}
Let $M$ be a minimal hypersurface in $\mathbb S^{n+1}$ with second fundamental form $A$. In 1968, Simons \cite{Sim68} derived the famous Simons' identity:
\begin{align}\label{Simonsiden}
\dfrac{1}{2}\Delta \abs{A}^2=\abs{\nabla A}^2 +\kh{n-\abs{A}^2}\abs{A}^2,
\end{align}
where $\nabla A$ is the covariant derivation of $A$. Therefore, if $M$ is closed and assume $0\le \abs{A}^2 \le n$, then the divergence theorem gives $\abs{A}^2\equiv 0$ or $\abs{A}^2\equiv n$. 
This pinching phenomenon has attracted  many mathematicians' attention. Chern-do Carmo-Kobayashi \cite{ChedoCKob70} and Lawson \cite{Law69}  independently characterized the case of $\abs{A}^2\equiv n$: the Clifford hypersurfaces are the only minimal hypersurfces in $\mathbb S^{n+1}$ with $\abs{A}^2\equiv n$. Assume $R_M$ is the scalar curvature of $M$. By the Gauss equation, one has 
$\abs{A}^2=n(n-1)+R_M$. Therefore, $M$ has constant scalar curvature iff $\abs{A}^2$ is constant. The authors \cite{ChedoCKob70} proposed the following conjecture:
\begin{conj}\label{Chernconj}
Let $M^n$ be a  closed  immersed  minimal  hypersurface of the unit sphere $\mathbb{S}^{n+1}$ with constant scalar curvature $R_M$. Then for each $n$, the set of all possible values for $R_M$  is discrete.
\end{conj}
Now this conjecture is known as {\it Chern conjecture}.
In 1983, Peng and Terng \cite{PenTer83, PenTer83a} made the first breakthrough towards Chern  Conjecture, they showed:
\emph{If $\abs{A}^2 > n$, then $\abs{A}^2 > n+ \frac{1}{12n}$. Moreover, for $n = 3$,  $\abs{A}^2 \ge  6$ if $\abs{A}^2 > 3$. } 
In 1993, Chang \cite{Cha93}  completed the proof of Chern  Conjecture for $n=3$. After that,
 Yang-Cheng \cite{YanChe98}  and  
 Suh-Yang  \cite{SuhYan07}  improved the constant from $\frac{1}{12n}$ to $\frac{3n}{7}$. Based on known examples, mathematicians believe this constant should be $n$, that is, if $\abs{A}^2\equiv const.>n$, then $\abs{A}^2\ge 2n$ (this is usually called {\it the second pinching conjecture}).  Up to now, even the second pinching conjecture is still open for higher dimension.

Note that, in Simons' pinching result, $\abs{A}^2$ is not required  to be constant. Thus, many mathematicians believe the second pinching conjecture still holds without the assumption that  $\abs{A}^2$ is constant. In this case, Peng-Terng \cite{PenTer83, PenTer83a} showed that there does exist a pinching phenomenon if $\abs{A}^2>n$ and $n\le 5$.
Later, Cheng-Ishikawa \cite{CheIsh99},  Wei-Xu \cite{WeiXu07} and Zhang \cite{Zha10}  promoted it to $n\leq8$. 
 Finally,   Ding-Xin \cite{DinXin11} proved  that if  $\abs{A}^2>n$ then $\abs{A}^2>n+\frac{n}{23}$ for all dimension $n$.  After that,  Xu-Xu \cite{XuXu17} and  Lei-Xu-Xu \cite{LeiXuXu17} improved the pinching constant $\frac{n}{23}$ to $\frac{n}{18}$. 

Since all the known  minimal hypersurfaces in $\mathbb S^{n+1}$ with constant scalar curvature are isoparametric (equivalently to have constant principal curvatures), mathematicians proposed the following {\it strong version of Chern conjecture}:
\begin{conj}\label{StrongversionChernconj}
A  closed  immersed  minimal  hypersurface $M^n$ of $\mathbb{S}^{n+1}$ with constant scalar curvature must be isoparametric.
\end{conj}
Actually, if $M$ is a minimal isoparametric hypersurface in $\mathbb S^{n+1}$, according to M\"{u}nzner's result \cite{Mun80}, $\abs{A}^2$  can only take values in the set $\dkh{0, n, 2n, 3n, 5n}$. It is worth pointing out that,
isoparametric hypersurfaces in $\mathbb S^{n+1}$ are completely classified until recently (see \cite{CecChiJen07, Chi11, Chi13, Chi20, DorNeh85, Imm08, Miy13,Miy16}). When $n=3$, Chang \cite{Cha93} gave an affirmative answer to Conjecture \ref{StrongversionChernconj}. 
When $n=4$, 
Deng-Gu-Wei \cite{DenGuWei17} proved Conjecture \ref{StrongversionChernconj} with an additional assumption that $M^4$ is Willmore. Note that in this case, $M$ is Willmore iff $f_3\equiv 0$ (see \cite{Li01}), where $f_r$ denotes the $r$-th power sum of the principal curvatures for each positive integer $r$. With assumption that  $R_M\geq0 $,  $f_3$ and  the number of distinct principal curvatures $g$ are constant, Tang and Yang \cite{TanYan18} proved Conjecture \ref{StrongversionChernconj}. Recently, Li \cite{Li22} showed Conjecture \ref{StrongversionChernconj} is true  
if $f_3$ is constant satisfies $f_3^2\le 72$ and the Gauss curvature $\mathcal{K}$ satisfies $\mathcal{K}\le 1$.
For higher dimensional case, with assumptions that the Gauss curvature $\mathcal{K}$ is constant and $M$ has three pairwise distinct principal curvatures everywhere, de Almeida-Brito-Scherfner-Weiss \cite{deABriSchWei18} showed that  Conjecture \ref{StrongversionChernconj} is true. Recently, Tang-Wei-Yan  \cite{TanWeiYan18} and  Tang-Yan \cite{TanYan23} showed Conjecture \ref{StrongversionChernconj} is true if $f_r$ is constant for each $k=1, \cdots, n-1$ and $R_M\ge 0$. Their results strongly support Conjecture \ref{StrongversionChernconj} and generalized de Almeida-Brito's results in \cite{deABri90}.

In this paper, we focus on the case of dimension $n=4$ and $M$ has vanishing Gauss curvature. In this case, a minimal isoparametric hypersurface must be totally geodesic (see Lemma \ref{lem1} in Section 3). Therefore,  based on the Conjecture \ref{StrongversionChernconj}, a natural question is: Is a closed minimal hypersurface with constant scalar curvature and zero Gauss curvature totally geodesic? We give an affirmative answer to this question as follows:
\begin{theorem}\label{mainthm}
Let $M$ be a closed minimal hypersurface in $\mathbb S^5$ with constant scalar curvature and zero Gauss curvature, then $M$ is totally geodesic.
\end{theorem}
We emphasize that without the assumption of constant scalar curvature, there are many non-totally geodesic minimal hypersurfaces in a unit sphere with zero Gauss curvature, cf. \cite{Ram90}.

This paper is arranged as follows. In Section 2, we  set notations and review some known formulas and results. In Section 3, we give the proof of Theorem \ref{mainthm}.

\vspace{1cm}
\section{Preliminary}
In this section, we will fix some notations and list several known formulas.

Let $M$ be a minimal hypersurface in $\mathbb S^5$ with second fundamental form $A$. We choose a local orthonormal frame $\dkh{e_i}_{i=1}^5$ in $\mathbb S^5$ such that $e_1, \cdots, e_4$ are tangent to $M$. Let $\dkh{\omega^i}_{i=1}^5$ be the dual frame of $\dkh{e_i}_{i=1}^5$.  Then $A$ can be written as $A=\sum_{i,j=1}^4h_{ij}\omega^i\otimes \omega^j$ with $h_{ij}=h_{ji}$. 
Unless otherwise specified, throughout this paper, the summation is always from 1 to 4.
The minimality of $M$ implies $0=H=\dfrac{1}{4}\sum_i h_{ii}$.
Set 
\begin{align*}
F_{,i}=\nabla_i F, \ \ F_{,ij}=\nabla_j\nabla_i F, \ \ h_{ijk}=\nabla_k h_{ij} \ \text{and} \  h_{ijkl}=\nabla_l\nabla_k h_{ij},
\end{align*}
where $\nabla_j$ is the covariant differentiation operator with respect to $e_j$. 
The Gauss equation, the Codazzi equation and the Ricci formulas are given by (cf. \cite{ChedoCKob70})
\begin{align}\label{Gausseq}
R_{ijkl}=&\delta_{ik}\delta_{jl}-\delta_{il}\delta_{jk}+h_{ik}h_{jl}-h_{il}h_{jk}.\\
h_{ijk}=&h_{ikj}.\label{Codazzieq}\\
h_{ijkl}-h_{ijlk}
=&\sum_m h_{im}R_{mjkl}+\sum_m h_{mj}R_{mikl}.\label{Ricciiden}
\end{align}
From the Gauss equation \eqref{Gausseq}, it follows that 
$
R_M=\sum_{i,j}R_{ijij}=12-\abs{A}^2.
$
Therefore, $\abs{A}^2$ is constant iff the scalar curvature is constant.
If $M$ has constant scalar curvature, Simons' identity \eqref{Simonsiden} implies
\begin{align}
\abs{\nabla A}^2=\abs{A}^2\kh{\abs{A}^2-4}.\label{nablaA}
\end{align}
Assume $A$ has principal curvature $\lambda_1,\cdots, \lambda_4$ at some point.
Combining \eqref{Gausseq} with  \eqref{Ricciiden} gives
\begin{align}
h_{ijkl}-h_{ijlk}=\kh{\lambda_i-\lambda_j}\kh{1+\lambda_i\lambda_j}\kh{\delta_{ik}\delta_{jl}-\delta_{il}\delta_{jk}}.\label{tijkl}
\end{align}

For each positive integer $k$, denote by 
\begin{align}\label{f_k}
f_k={\rm tr} A^k,\quad \sigma_k=\sum_{i_1<\cdots<i_k}\lambda_{i_1}\cdots\lambda_{i_k}.
\end{align}
Although $f_k$ and $\sigma_k$ are defined by using the principal curvatures, they are actually globally well defined smooth functions.
A direct computation gives,
\begin{align}
\begin{cases}\label{f1234}
&f_1 = \sigma_1=4H=0, \\ 
&f_2 =\sigma_1^2-2\sigma_2=\abs{A}^2,\\
&f_3 = \sigma_1^3-3\sigma_1\sigma_2+3\sigma_3
=3\sigma_3,\\
&f_4 = \sigma_1^4-4\sigma_1^2\sigma_2
+4\sigma_1\sigma_3+2\sigma_2^2-4\sigma_4=\frac{1}{2}\abs{A}^4-4\mathcal{K},
\end{cases}
\end{align}
where $\mathcal{K}=\sigma_4=\lambda_1\lambda_2\lambda_3
\lambda_4$
is the Gauss curvature (also called Gauss-Kronecker curvature). Note that 
\begin{align}\label{K=0}
\mathcal{K}\equiv 0 \ \ \text{iff}\ \ f_4\equiv\dfrac{1}{2}\abs{A}^4.
\end{align}
The following identities are due to Peng and Terng \cite{PenTer83},
\begin{align}
\Delta f_3=&3\kh{4-\abs{A}^2}f_3+6\mathscr{C}, \label{Deltaf3}\\
\Delta f_4 =& 4\kh{4-\abs{A}^2}f_4 +4\kh{2\mathscr{A}+\mathscr{B}},\label{Deltaf4}
\end{align}
where 
\begin{align}\label{AandB}
\mathscr{A}=\sum_{i,j,k} \lambda_i^2 h_{ijk}^2,\quad \mathscr{B}=\sum_{i,j,k} \lambda_i\lambda_jh_{ijk}^2,\quad \mathscr{C}=\sum_{i,j,k}\lambda_ih_{ijk}^2.
\end{align}

\vspace{1cm}
\section{Proof of the main theorem}

The following result is due to Cartan \cite{Car38}.
\begin{lem}[{cf. \cite[page 84]{BerConOlm03}}]
Let $M$ be an isoparametric hypersurface in $\mathbb S^{n+1}$. Denote by $g$ the number of distinct principal curvatures, and by $\mu_1, \cdots, \mu_g$ the distinct principal curvatures with corresponding multiplicities $m_1, \cdots, m_g$. Then for each fixed $i=1,\cdots, g$,
\begin{align}\label{Cartanformula}
\sum_{j=1,\  j\neq i}^{g} m_j\dfrac{1+\mu_i\mu_j}{\mu_i-\mu_j}=0.
\end{align}

\end{lem}

Although isopermetric hypersurfaces are completely classified, we can give a direct proof of the following result by Cartan's fundamental formula.
\begin{lem}\label{lem1}
Assume $M$ is a minimal isoparametric hypersurface in $\mathbb S^5$ with zero Gauss curvature, then $M$ is totally geodesic.
\end{lem}
\begin{proof}
Assume $\lambda_1\le \lambda_2\le \lambda_3\le \lambda_4$ are the principal curvatures of $M$. By our assumption,
\begin{align*}
0=4H =\sum_{i=1}^4 \lambda_i, \quad 0=\mathcal{K}=\lambda_1\lambda_2\lambda_3 \lambda_4.
\end{align*}

Suppose $M$ is not totally geodesic,   we can assume, without loss of generality,  $\lambda_2=0$. Then the principal curvatures are
\begin{align*}
\lambda_1<\lambda_2=0\le \lambda_3 \le \lambda_4>0.
\end{align*}
There are three possibilities:
\begin{itemize}
\item $\lambda_3=0$. In this case, we have 
\begin{align*}
g=3, \mu_1=\lambda_1, \mu_2=0, \mu_3=\lambda_4=-\lambda_1, m_1=1, m_2=2, m_3=1.
\end{align*}
Therefore, for $i=1$ in \eqref{Cartanformula}, we have
\begin{align*}
\dfrac{2}{\lambda_1}+\dfrac{1-\lambda_1^2}{2\lambda_1}=0,
\end{align*}
which implies $\lambda_1=-\sqrt{5}$. Consequently,  $\abs{A}^2=\lambda_1^2+\lambda_4^2=10$. However, from M\"{u}nzner's result \cite{Mun80}, $\abs{A}^2=(g-1)n=2\times 4=8$, we get a contradiction.

\item $\lambda_3=\lambda_4>0$. In this case, we have
\begin{align*}
g=3, \mu_1=\lambda_1, \mu_2=0, \mu_3=\lambda_3=\lambda_4=-\dfrac{1}{2}\lambda_1, m_1=1, m_2=1, m_3=2.
\end{align*}
Therefore, for $i=2$ in \eqref{Cartanformula}, we obtain,
\begin{align*}
0=\dfrac{1}{0-\lambda_1}+2\dfrac{1}{0-\kh{-\frac{1}{2}\lambda_1}}=\dfrac{3}{\lambda_1},
\end{align*}
which is impossible.

\item $0<\lambda_3<\lambda_4$. In this case, we have
\begin{align*}
g=4, \mu_i=\lambda_i, m_i=1, i=1,\cdots, 4.
\end{align*}
Therefore, for $i=2$ in \eqref{Cartanformula}, we obtain,
\begin{align*}
0=\dfrac{1}{0-\lambda_1}+\dfrac{1}{0-\lambda_3}+\dfrac{1}{0-\lambda_4}=-\dfrac{\lambda_1\lambda_3+\lambda_1\lambda_4+\lambda_3\lambda_4}{\lambda_1\lambda_3\lambda_4},
\end{align*}
which implies $\lambda_1\lambda_3+\lambda_1\lambda_4 +\lambda_3\lambda_4=0$. Since $\lambda_1+\lambda_3+\lambda_4=0,$ we have
\begin{align*}
\lambda_3^2+\lambda_3\lambda_4 +\lambda_4^2=0,
\end{align*}
which is impossible for $\lambda_4>\lambda_3>0$.
\end{itemize}
Therefore, $M$ must be totally geodesic.
\end{proof}

We also need an elementary algebraic lemma.
\begin{lem}\label{algebraiclem}
Assume $a,b,c\in \mathbb R$ satisfying $a\le b\le c$ and $a+b+c=0$, then 
\begin{align*}
-\dfrac{1}{\sqrt{6}}\kh{a^2+b^2+c^2}^{3/2}\le a^3+b^3+c^3\le \dfrac{1}{\sqrt{6}}\kh{a^2+b^2+c^2}^{3/2},
\end{align*}
and 
\begin{align*}
a^3+b^3+c^3=&-\dfrac{1}{\sqrt{6}}\kh{a^2+b^2+c^2}^{3/2} \ \ \text{iff}\ \ b=c=-\dfrac{a}{2},\\
a^3+b^3+c^3=& \dfrac{1}{\sqrt{6}}\kh{a^2+b^2+c^2}^{3/2} \ \ \text{iff}\ \ a=b=-\dfrac{c}{2},\\
a^3+b^3+c^3=&0 \ \ \text{iff} \ \ b=0 \ \ \text{and} \ \ a=-c.
\end{align*}
\end{lem}

\begin{lem}\label{lem2}
Let $M$ be a closed minimal hypersurface in $\mathbb S^5$ with constant scalar curvature and zero Gauss-Kronecker curvature. If the smallest (or largest) principal curvature $\lambda_1$ has multiplicity two at some point $p$, then $M$ is totally geodesic.
\end{lem}
\begin{proof}
Assume $M$ is not totally geodesic, but the multiplicity of the smallest principal curvature $\lambda_1$ is two at $p$.
Since $\mathcal{K}=0$, then at  $p$ we can assume the principal curvatures are
\begin{align}
\lambda_1=\lambda_2=-\lambda<0=\lambda_3<\lambda_4=2\lambda>0.\label{lambda1multi2}
\end{align}
At the point $p$, by Lemma \ref{algebraiclem}, $f_3=\frac{1}{\sqrt{6}}\abs{A}^3$ attains $f_3$'s maximum. Therefore at $p$, we have for each $k$
\begin{align}
\begin{cases}
&H_{,k}=0;\\
&\kh{\abs{A}^2}_{,k}=0;\\
&\kh{f_3}_{,k}=0;\\
&\kh{f_4}_{,k}=0,
\end{cases}
\ \ \text{which yields}\ \ 
\begin{cases}
&\sum_i h_{iik}=0,\\
&\sum_i \lambda_i h_{iik}=0,\\
&\sum_i \lambda_i^2 h_{iik}=0,\\
&\sum_i \lambda_i^3 h_{iik}=0,
\end{cases} \label{lineareq1}
\end{align}
Consider that the value of principal curvatures \eqref{lambda1multi2}, we have 
\begin{align*}
h_{11k}+h_{22k}=0, \ \ h_{44k}=h_{33k}=0.
\end{align*}
We have 
\begin{align*}
\abs{\nabla A}^2=&4\kh{h_{111}^2+h_{112}^2}+6\kh{h_{113}^2+h_{114}^2+h_{123}^2+h_{124}^2+h_{134}^2+h_{234}^2}.\\
\mathscr{A}=&\sum_{i,j,k}\lambda_i^2 h_{ijk}^2\\
=&\lambda^2\sum_{j,k}\kh{h_{1jk}^2+h_{2jk}^2+4h_{4jk}^2}\\
=&\dfrac{1}{6}\abs{A}^2\zkh{4\kh{h_{111}^2+h_{112}^2+h_{113}^2+h_{123}^2}+10\kh{h_{234}^2+h_{134}^2}+12\kh{h_{114}^2+h_{124}^2}}.\\
\mathscr{B}=&\sum_{i,j,k}\lambda_i\lambda_jh_{ijk}^2\\
=&\lambda^2\sum_k\kh{h_{11k}^2+h_{22k}^2+2h_{12k}^2-4h_{14k}^2-4h_{24k}^2}\\
=&\lambda^2\zkh{2\kh{h_{113}^2+h_{123}^2}+4\kh{h_{111}^2+h_{112}^2-h_{134}^2-h_{234}^2}-6\kh{h_{114}^2+h_{124}^2} }.
\end{align*}
Therefore, combined  the above formula with \eqref{Deltaf4}, we have,
\begin{align*}
&3\zkh{4\kh{h_{111}^2+h_{112}^2}+6\kh{h_{113}^2+h_{114}^2+h_{123}^2+h_{124}^2+h_{134}^2+h_{234}^2}}\\
=&
3\abs{\nabla A}^2\\
=&3\abs{A}^2\kh{\abs{A}^2-4}\\
=&\dfrac{1}{\lambda^2}\kh{2\mathscr{A}+\mathscr{B}}\\
=&12\kh{h_{111}^2+h_{112}^2}+10\kh{h_{113}^2+h_{123}^2}+16\kh{h_{234}^2+h_{134}^2}+18\kh{h_{114}^2+h_{124}^2},
\end{align*}
which implies
\begin{align*}
0=8\kh{h_{113}^2+h_{123}^2}+2\kh{h_{134}^2+h_{234}^2}.
\end{align*}
Consequently,
\begin{align*}
h_{113}=h_{123}=h_{134}=h_{234}=0.
\end{align*}
Moreover,  since $H, \abs{A}^2$ and  $f_4$ are all constants,
we have, for all $1\le k,l\le 4$,
\begin{align*}
\begin{cases}
&H_{,kl}=0,\\
&\kh{\abs{A}^2}_{,kl}=0,\\
&\kh{f_4}_{,kl}=0.
\end{cases}
\end{align*}
A direct computation yields,\begin{align}\label{hijkl}
\begin{cases}
&\sum_i h_{iikl}=0,\\
&\sum_{i,j} \kh{h_{ijk}h_{ijl}+h_{ij}h_{ijkl}}=0,\\
&\sum_{i}\kh{\lambda_i^3h_{iikl}}+\sum_{i,j}\kh{2\lambda_i^2h_{ijk}h_{ijl}+\lambda_i\lambda_jh_{ijk}h_{ijl}}=0.
\end{cases}
\end{align}
Note that, at the point $p$, we have \eqref{lambda1multi2}
and
\begin{align*}
h_{11k}+h_{22k}=0, h_{33k}=h_{44k}=0, h_{113}=h_{123}=h_{134}=h_{234}=0.
\end{align*}
Therefore, for $k=l=3$, \eqref{hijkl} becomes,
\begin{align*}
\begin{cases}
&h_{1133}+h_{2233}+h_{3333}+h_{4433}=0,\\
&-\lambda h_{1133}-\lambda h_{2233}+2\lambda h_{4433}=0,\\
&-\lambda^3 h_{1133}-\lambda^3 h_{2233}+8\lambda^3 h_{4433}=0,
\end{cases}
 \end{align*}
which implies $h_{4433}=0$.
While for $k=l=4$, \eqref{hijkl} becomes,
\begin{align*}
&\begin{cases}
&h_{1144}+h_{2244}+h_{3344}+h_{4444}=0,\\
&-\lambda h_{1144}-\lambda h_{2244}+2\lambda h_{4444}=-2\kh{h_{114}^2+h_{124}^2},\\
&-\lambda^3 h_{1144}-\lambda^3 h_{2244}+8\lambda^3 h_{4444}=-6\lambda^2 \kh{h_{114}^2+h_{124}^2},
\end{cases}
\end{align*}
which implies $h_{3344}=0$. Consequently, 
\begin{align*}
h_{3344}-h_{4433}=0.
\end{align*}
However, \eqref{tijkl} and \eqref{lambda1multi2} yield 
\begin{align*}
h_{3344}-h_{4433}=-2\lambda<0.
\end{align*}
We get a contradiction.
\end{proof}

In dimension four, the Gauss-Bonnet-Chern formula is quite special. In our case, we have the following lemma.
\begin{lem}\label{GBClem}
Assume $M$ is a closed minimal hypersurface in $\mathbb S^5$ with second fundamental form $A$, Euler number $\chi(M)$,  and $f_4$ is defined by \eqref{f_k}. Then the Gauss-Bonnet-Chern formula becomes
\begin{align}\label{GBC}
\int_M \kh{\dfrac{3}{2}\abs{A}^4 -3f_4-2\abs{A}^2+12}=16\pi^2\chi\kh{M}.
\end{align}
\end{lem}
\begin{proof}
For a closed 4-dimensional manifold, the Gauss-Bonnet-Chern formula (see \cite{Ave63} or \cite{Bes87}) reads
\begin{align}\label{GBC2nd}
\int_M \kh{\frac{R_M^2}{3} -\abs{\rm Ric}^2+\frac{\abs{W}^2}{2}}=16\pi^2 \chi(M),
\end{align}
where $R_M, Ric, W$ are the scalar curvature, Ricci tensor, Weyl curvature tensor of $M$ respectively, which can be calculated as follows (see \cite[p.117]{Aub98}),
\begin{align*}
R_M=&\sum_{i,j} R_{ijij},\\
\abs{{\rm Ric}}^2 =&\sum_{i,j} R_{ij}^2=:\sum_{i,j} \kh{\sum_k R_{ikjk}}^2,\\
\abs{W}^2=& \sum_{i,j,k,l} W_{ijkl}^2\\
=&\sum_{i,j,k,l} \zkh{R_{ijkl}-\dfrac{1}{2}\kh{R_{ik}\delta_{jl}-R_{il}\delta_{jk}+R_{jl}\delta_{ik}-R_{jk}\delta_{il}}+\dfrac{R_M}{6}\kh{\delta_{ik}\delta_{jl}-\delta_{il}\delta_{jk}}}^2.
\end{align*}
 Then combined the above formulas and Gauss equation \eqref{Gausseq}, a direct computation yields
\begin{align*}
R_M=&12-\abs{A}^2,\\
\abs{Ric}^2=&36-6\abs{A}^2+f_4,\\
\abs{W}^2=&\dfrac{7}{3}\abs{A}^4-4f_4.
\end{align*}
Substituting the above formulas into Gauss-Bonnet-Chern formula \eqref{GBC}, we obtain \eqref{GBC2nd}.
\end{proof}

\noindent{\bf Remark.}
Since for $n\ge 4$, locally conformally flat is equivalent to $W\equiv 0$. Therefore, by the proof of the above lemma, a minimal hypersurface in $\mathbb S^5$ is locally conformally flat iff $\abs{A}^4=\dfrac{12}{7}f_4$. By a direct computation, the trace-free Ricci tensor $\mathring{Ric}=Ric -\dfrac{R_M}{4}Id$ satisfies
\begin{align*}
\abs{\mathring{Ric}}^2=f_4 -\dfrac{1}{4}\abs{A}^4.
\end{align*}
Hence, $M$ is Einstein iff $\mathring{Ric}\equiv 0$, i.e., $4f_4 =\abs{A}^4$.

We also need a result due to Cheng-Yang.
\begin{lem}[{\cite[Theorem 2]{YanChe98}}]\label{ChengYanglemma}
Let $M^4$ be a closed minimal hypersurface of $S^{5}$ with constant scalar curvature and constant $f_4$. If $\abs{A}^2>4$, then $\abs{A}^2 \ge \frac{20}{3}$.
\end{lem}
\noindent{\bf Remark.} Cheng-Yang actually proved the above lemma for general dimension $n$ and with the assumption that $f_3$ is constant instead of $f_4$ is constant.  The key ingredients of their proof are equalities \eqref{Deltaf3} and \eqref{Deltaf4}. When $f_4$ is constant, we can change a little in the proof of \cite[Lemma 3.1]{YanChe98}, such that 
\begin{align*}
0=&3\kh{4-\abs{A}^2}f_3+6\mathscr{C},  \ \ \text{at a point} \ \ x_0\in M,\\
0=& 4\kh{4-\abs{A}^2}f_4 +4\kh{2\mathscr{A}+\mathscr{B}}, \ \ \text{at all points in }\ M.
\end{align*}
Consider at $x_0$, we can get the same result as \cite[Lemma 3.1]{YanChe98}, and consequently we can obtain Lemma \ref{ChengYanglemma}. Cheng-Yang also mentioned this lemma in \cite[Remark 1]{YanChe98}.

Now we can give the proof of Theorem \ref{mainthm}.

\begin{proof}[Proof of Theorem \ref{mainthm}] Suppose $M$ is not totally geodesic, then from Lemma \ref{lem2}, we know the smallest principal curvature $\lambda_1$ can not has multiplicity two at any point. Since $\mathcal{K}=\lambda_1\lambda_2\lambda_3\lambda_4=0$, $\lambda_1$ can not has multiplicity 3 at any point either. Therefore, $\lambda_1$ always has multiplicity one.  Consequently, the unit eigenvectors of
$\lambda_1$ form a nowhere vanishing vector field on $M$, which implies $\chi(M)=0$ by the Poincar\'{e}-Hopf theorem. Combined this fact with $\mathcal{K}\equiv 0$ and relation \eqref{K=0}, the Gauss-Bonnet-Chern formula \eqref{GBC} is reduced to
\begin{align*}
\int_M \kh{\abs{A}^2-6}=0,
\end{align*}
which implies $\abs{A}^2=6$. 

On the other hand, since $M$ has constant scalar curvature Simons' identity implies $\abs{A}^2\ge 4$. However, $\abs{A}^2=4$ iff $M$ is a Clifford hypersurface (\cite{ChedoCKob70, Law69}), which is impossible since the Clifford hypersurfaces have non-zero Gauss curvature. Thus, we have $\abs{A}^2>4$. Notice that, $\mathcal{K}=0$ implies $f_4=\frac{1}{2}\abs{A}^4$ is constant. Consequently, we have $\abs{A}^2\ge\frac{20}{3}$ from Lemma \ref{ChengYanglemma}.  We get a contradiction.
\end{proof}

\vspace{1cm}

\providecommand{\bysame}{\leavevmode\hbox to3em{\hrulefill}\thinspace}
\providecommand{\MR}{\relax\ifhmode\unskip\space\fi MR }


\end{document}